\documentclass[12pt]{amsart}
\usepackage{amsmath,amssymb,amsbsy,amsfonts,amsthm,latexsym,
                        amsopn,amstext,amsxtra,euscript,amscd,mathrsfs,color,bm}
                        
\usepackage{float} 
\usepackage[english]{babel}
\usepackage{url}
\usepackage[colorlinks,linkcolor=blue,anchorcolor=blue,citecolor=blue,backref=page]{hyperref}

\begin{document}

 \renewcommand*{\backref}[1]{}
\renewcommand*{\backrefalt}[4]{%
    \ifcase #1 (Not cited.)%
    \or        (p.\,#2)%
    \else      (pp.\,#2)%
    \fi}     
    
\newtheorem{theorem}{Theorem}
\newtheorem{lemma}[theorem]{Lemma}
\newtheorem{example}[theorem]{Example}
\newtheorem{algol}{Algorithm}
\newtheorem{corollary}[theorem]{Corollary}
\newtheorem{prop}[theorem]{Proposition}
\newtheorem{definition}[theorem]{Definition}
\newtheorem{question}[theorem]{Question}
\newtheorem{problem}[theorem]{Problem}
\newtheorem{remark}[theorem]{Remark}
\newtheorem{conjecture}[theorem]{Conjecture}

\newcommand{\commI}[1]{\marginpar{%
\begin{color}{magenta}
\vskip-\baselineskip 
\raggedright\footnotesize
\itshape\hrule \smallskip I: #1\par\smallskip\hrule\end{color}}}

\newcommand{\commM}[1]{\marginpar{%
\begin{color}{red}
\vskip-\baselineskip 
\raggedright\footnotesize
\itshape\hrule \smallskip M: \bf #1\par\smallskip\hrule\end{color}}}

\newcommand{\commA}[1]{\marginpar{%
\begin{color}{blue}
\vskip-\baselineskip 
\raggedright\footnotesize
\itshape\hrule \smallskip #1\par\smallskip\hrule\end{color}}}

\def\xxx{\vskip5pt\hrule\vskip5pt}

\def\Cmt#1{\underline{{\sl Comments:}} {\it{#1}}}

\newcommand{\Modp}[1]{
\begin{color}{blue}
 #1\end{color}}


\def\cA{{\mathcal A}}
\def\cB{{\mathcal B}}
\def\cC{{\mathcal C}}
\def\cD{{\mathcal D}}
\def\cE{{\mathcal E}}
\def\cF{{\mathcal F}}
\def\cG{{\mathcal G}}
\def\cH{{\mathcal H}}
\def\cI{{\mathcal I}}
\def\cJ{{\mathcal J}}
\def\cK{{\mathcal K}}
\def\cL{{\mathcal L}}
\def\cM{{\mathcal M}}
\def\cN{{\mathcal N}}
\def\cO{{\mathcal O}}
\def\cP{{\mathcal P}}
\def\cQ{{\mathcal Q}}
\def\cR{{\mathcal R}}
\def\cS{{\mathcal S}}
\def\cT{{\mathcal T}}
\def\cU{{\mathcal U}}
\def\cV{{\mathcal V}}
\def\cW{{\mathcal W}}
\def\cX{{\mathcal X}}
\def\cY{{\mathcal Y}}
\def\cZ{{\mathcal Z}}

\def\C{\mathbb{C}}
\def\F{\mathbb{F}}
\def\K{\mathbb{K}}
\def\L{\mathbb{L}}
\def\G{\mathbb{G}}
\def\Z{\mathbb{Z}}
\def\R{\mathbb{R}}
\def\Q{\mathbb{Q}}
\def\N{\mathbb{N}}
\def\M{\textsf{M}}
\def\U{\mathbb{U}}
\def\P{\mathbb{P}}
\def\A{\mathbb{A}}
\def\fp{\mathfrak{p}}
\def\n{\mathfrak{n}}
\def\X{\mathcal{X}}
\def\x{\textrm{\bf x}}
\def\w{\textrm{\bf w}}
\def\ovQ{\overline{\Q}}
\def \Kab{\K^{\mathrm{ab}}}
\def \Qab{\Q^{\mathrm{ab}}}
\def \Qtr{\Q^{\mathrm{tr}}}
\def \Kc{\K^{\mathrm{c}}}
\def \Qc{\Q^{\mathrm{c}}}
\def\rank#1{\mathrm{rank}#1}
\def\ZK{\Z_\K}
\def\ZKS{\Z_{\K,\cS}}
\def\ZKSf{\Z_{\K,\cS_f}}

\def\({\left(}
\def\){\right)}
\def\[{\left[}
\def\]{\right]}
\def\<{\langle}
\def\>{\rangle}

\def\gen#1{{\left\langle#1\right\rangle}}
\def\genp#1{{\left\langle#1\right\rangle}_p}
\def\genPs{{\left\langle P_1, \ldots, P_s\right\rangle}}
\def\genPsp{{\left\langle P_1, \ldots, P_s\right\rangle}_p}

\def\e{e}

\def\eq{\e_q}
\def\fh{{\mathfrak h}}

\def\lcm{{\mathrm{lcm}}\,}

\def\({\left(}
\def\){\right)}
\def\fl#1{\left\lfloor#1\right\rfloor}
\def\rf#1{\left\lceil#1\right\rceil}
\def\mand{\qquad\mbox{and}\qquad}

\def\jt{\tilde\jmath}
\def\ellmax{\ell_{\rm max}}
\def\llog{\log\log}

\def\m{{\rm m}}
\def\ch{\hat{h}}
\def\GL{{\rm GL}}
\def\Orb{\mathrm{Orb}}
\def\Per{\mathrm{Per}}
\def\Preper{\mathrm{Preper}}
\def \S{\mathcal{S}}
\def\vec#1{\mathbf{#1}}
\def\ov#1{{\overline{#1}}}
\def\Gal{{\mathrm Gal}}
\def\Sp{{\mathrm S}}
\def\tors{\mathrm{tors}}
\def\PGL{\mathrm{PGL}}
\def\wH{{\rm H}}
\def\Gm{\G_{\rm m}}

\newcommand{\bfalpha}{{\boldsymbol{\alpha}}}
\newcommand{\bfomega}{{\boldsymbol{\omega}}}

\newcommand{\Ch}{{\operatorname{Ch}}}
\newcommand{\Elim}{{\operatorname{Elim}}}
\newcommand{\proj}{{\operatorname{proj}}}
\newcommand{\h}{{\operatorname{\mathrm{h}}}}
\newcommand{\ord}{\operatorname{ord}}

\newcommand{\hh}{\mathrm{h}}
\newcommand{\aff}{\mathrm{aff}}
\newcommand{\Spec}{{\operatorname{Spec}}}
\newcommand{\Res}{{\operatorname{Res}}}

\def\fA{{\mathfrak A}}
\def\fB{{\mathfrak B}}

\def\bphi{\pmb{\varphi}}

\def\house#1{{%
    \setbox0=\hbox{$#1$}
    \vrule height \dimexpr\ht0+1.4pt width .5pt depth \dp0\relax
    \vrule height \dimexpr\ht0+1.4pt width \dimexpr\wd0+2pt depth \dimexpr-\ht0-1pt\relax
    \llap{$#1$\kern1pt}
    \vrule height \dimexpr\ht0+1.4pt width .5pt depth \dp0\relax}}

\numberwithin{equation}{section}
\numberwithin{theorem}{section}

\title[Multiplicative dependence of rational functions]
{On multiplicative dependence of values of rational functions and a generalisation 
of the Northcott theorem}

\author[A. Ostafe] {Alina Ostafe}
\address{School of Mathematics and Statistics, University of New South Wales, Sydney NSW 2052, Australia}
\email{alina.ostafe@unsw.edu.au}

\author[M. Sha]{Min Sha}
\address{Department of Computing, Macquarie University, Sydney, NSW 2109,
Australia}
\email{shamin2010@gmail.com}

\author[I. Shparlinski] {Igor E. Shparlinski}
\address{School of Mathematics and Statistics, University of New South Wales, Sydney NSW 2052, Australia}
\email{igor.shparlinski@unsw.edu.au}

\author[U. Zannier]{Umberto Zannier}
\address{Scuola Normale Superiore, Piazza dei Cavalieri, 7, 56126 Pisa, Italy}
\email{u.zannier@sns.it}



\begin{abstract} 
In this paper we study multiplicative dependence of values of polynomials or rational functions over a number field. 
As an application, we obtain new results on multiplicative dependence in the orbits of a univariate 
polynomial dynamical system. We also obtain a broad generalisation of the  Northcott theorem
replacing the finiteness of preperiodic points from a given number field  by   the finiteness of initial points 
with two multiplicatively dependent elements in their orbits.  
\end{abstract}

\maketitle

\section{Introduction}
\subsection{Motivation and background} 

We say that non-zero complex numbers $\alpha_1,\ldots,\alpha_n$ are \textit{multiplicatively dependent} 
if there exist integers $k_1,\ldots,k_n$, not all zero, such that 
$$
\alpha_1^{k_1} \cdots \alpha_n^{k_n}=1. 
$$
Consequently, a point in the complex space $\C^n$ is called \textit{multiplicatively dependent} if its coordinates are all non-zero and  are multiplicatively dependent. 

The same definition of multiplicative dependence applies to rational functions as well. 
Moreover, we say that rational functions $\varphi_1,\ldots,\varphi_s\in\C(X)$ are   \textit{multiplicatively independent modulo constants} 
if there is no non-zero integer vector $(k_1,\ldots,k_s)$ such that 
$$
\varphi_1^{k_1} \cdots \varphi_s^{k_s} \in \C^*.
$$

We also use  $\Gm$ to denote  the multiplicative algebraic group, that is $\Gm = \ov \Q^*$ endowed with the multiplicative group law, 
where as usual $ \ov \Q$ denotes the algebraic closure of $\Q$. 
The study of intersections of geometrically irreducible  algebraic  curves $\cX \subseteq \Gm^n$, defined over a number field $\K$, and  a union of proper algebraic subgroups of $\Gm^n$ has been initiated by
Bombieri, Masser and  Zannier~\cite{BMZ} (see also~\cite{Zan2}).
It is well known (see, for example,~\cite[Corollary~3.2.15]{BoGu}) that  each such subgroup of $\Gm^n$ is defined by a finite set of equations of the shape $X_1^{k_1}\cdots X_n^{k_n}=1$, with integer exponents not all zero. 
That is, the work~\cite{BMZ} is about multiplicative dependence of points on a curve. 
It has been proved in~\cite[Theorem~1]{BMZ}  that, under the assumption that {\it $\cX$ is not contained in
any translate of a proper algebraic subgroup of $\Gm^n$}, the multiplicatively dependent points on $\cX(\overline{\Q})$ form a set of bounded (absolute logarithmic) Weil  height. 

Most recently, a new point of view has been introduced in~\cite{OSSZ}. 
In particular,  the structure of  multiplicatively dependent points on 
$\cX(\Kab)$ has been established in~\cite{OSSZ},  where $\Kab$ is the maximal abelian extension of a number field $\K$. 
In turn this implies that the set of such points is finite if $\cX$ is of positive genus 
(see~\cite[Remark 2.4]{OSSZ}\footnote{We take this opportunity to indicate that
the justification given in~\cite[Remark 2.4]{OSSZ} mentions the infinitude of the
automorphism group of $\Gm$ while  in fact it contains only two automorphisms; 
however  the claim itself is correct as the simple comparison of the genera of $\Gm$ and $\cX$ shows.}). However, other than the structure, this result does not tell us further information about the case when $\cX$ is of genus zero. 
Here we  address this issue  and give several applications to algebraic dynamical systems. 

In algebraic geometry (see, for instance,~\cite[Section~1.2]{KSC}), a rational curve defined over a field $\F$ is a curve birationally isomorphic to the projective line $\P^1$. 
If $\F$ is algebraically closed, this is equivalent to a curve of genus zero. 
Moreover, each rational curve can be parametrized in the form $(\varphi_1(X),\ldots,\varphi_s(X))$, where $\varphi_i\in \F(X)$ 
are rational functions, not all constant. Conversely, each curve of this form represents a rational curve. 
So, when considering multiplicatively dependent points on a curve of genus zero, we essentially study 
multiplicative dependence in values of some rational functions, which is exactly the topic of this paper. 
Since we only discuss multiplicative dependence for non-zero complex numbers, 
this setting automatically excludes poles and zeros of rational functions. 
See also~\cite{DubSha} for the case of translations of algebraic numbers. 

This work is also partially motivated by  a series of recent results on the distribution of roots of unity and, 
more generally, of  algebraic numbers of bounded house (we refer to Section~\ref{sec:def} for a precise definition) in orbits 
of rational functions.  
Here, we consider a broad generalisation of such problems and in particular  study the 
multiplicative dependence  of several consecutive iterations $\varphi^{(n+1)}(\alpha), \ldots, \varphi^{(n+s)}(\alpha)$.

Recall that for a rational function $\varphi\in\K(X)$, the $n$-th iterate  $\varphi^{(n)}$ of $\varphi$ is recursively defined by
$$
\varphi^{(0)}=X,\mand \varphi^{(n)}=\varphi\(\varphi^{(n-1)}\),\quad n\ge 1.
$$ 
For an element $\alpha\in \ov\Q$ we define the {\it orbit} of $\varphi$ at $\alpha$ as the set 
\begin{equation}
\label{eq:Orb h}
\Orb_\varphi(\alpha)= \{u_n: \, u_0=\alpha, \, u_n =\varphi(u_{n-1}),\,  n=1,2, \ldots \}. 
\end{equation}

\begin{remark} 
{\rm
The  iterations in the orbit $\Orb_\varphi(\alpha)$  are defined until some point $u_{n-1}$ hits a 
pole of $\varphi$. Moreover, if some point $u_{n}$, $n \ge 1$, in~\eqref{eq:Orb h} is defined, then
$\alpha$ is not a pole of $\varphi^{(n)}$ and $u_{n}=\varphi^{(n)}(\alpha)$. However, 
  the converse is not true, and the fact that the evaluation $\varphi^{(n)}(\alpha)$ is defined does not imply the existence of
$u_{n}$, since  it  is defined if and only if all the previous points $u_0, \ldots, u_{n-1}$ 
of the orbit~\eqref{eq:Orb h} are defined and $u_{n-1}$ is not a pole of $\varphi$.  
Clearly, for polynomial systems this distinction does not exist. 
}
\end{remark} 

\begin{definition} [\textit{Periodic and preperiodic points}]
\label{def:pep_preper}
{\rm
We say that $\alpha \in \overline{\Q}$ 
is a \textit{periodic point} for $\varphi$ if  $\varphi^{(n)}(\alpha)=\alpha$ for some positive 
integer $n$,   and we say that $\alpha$ is a \textit{preperiodic point}  for $\varphi$ if $\varphi^{(m)}(\alpha)$ is a periodic point for some positive integer $m$. 
}
\end{definition}

The  famous {\it Northcott theorem\/}, see~\cite[Theorem~3.12]{Silv}, 
gives the finiteness of preperiodic points for $\varphi$ contained in a number field and of bounded Weil height when $\varphi$ is of degree  at least 2. 
Here we obtain an extension of this 
finiteness result in a new direction, which involves the notion of \textit{multiplicative dependence}
instead of \textit{equality}.

\subsection{General notation and convention} 
\label{sec:def}

Throughout the paper, we use the following notation:
\begin{itemize}
\item[] $\ovQ$: the  algebraic  closure of the rational numbers $\Q$;
\item[ ] $\U$: the set of all roots of unity in the complex numbers $\C$; 
\item[] $\K$: an algebraic number field;
\item[] $d_{\K}$: the degree over $\Q$ of the number field $\K$;
\item[]  $\ZK$: the ring of integers of $\K$;
\item[] $\Kc=\K(\U)$: the maximal cyclotomic extension of  $\K$; 
\item[] $\Kab$: the maximal abelian extension of $\K$; 
\item[] $\bphi=(\varphi_1,\ldots,\varphi_s) \in \C(X)^s$: a vector of rational functions; 
\item[] $\bphi(\alpha)=(\varphi_1(\alpha),\ldots,\varphi_s(\alpha)) \in \C^s$ for $\alpha \in \C$. 
\end{itemize}

We recall that by the  {\it Kronecker--Weber\/} theorem, 
 $\Qc = \Qab$, see~\cite[Chapter~14]{Wash}.  However, generally we can 
 only claim $\Kc \subseteq \Kab$.

We  reserve $|\alpha|$ for  the usual  absolute value of  $\alpha \in \C$ and use
$\house{\strut\alpha}$ for the {\it house\/} of $\alpha$, which is the maximum of absolute values $|\sigma(\alpha)|$ 
of the conjugates $\sigma(\alpha)$ over $\Q$ of   $\alpha\in \ovQ$.

{\it Height\/} always means the {\it absolute logarithmic Weil height\/}  which we denote by $\h(\alpha)$ for non-zero   $\alpha\in \ovQ$,
see~\cite{BoGu,Zan1}. 

For a rational function $\varphi\in\K(X)$ with $\varphi=f/g$, $f,g\in \K[X]$ and $\gcd(f,g)=1$, 
we define the \textit{degree} of $\varphi$, denoted by $\deg \varphi$, to be $\max\{\deg f, \deg g\}$. 
We say that $\varphi$ is \textit{monic}, if both $f$ and $g$ are monic.

\begin{definition} [\textit{Special rational functions}]
\label{def:Srf}
{\rm
We say that a rational function $\varphi\in\K(X)$ of degree $d$  is \textit{special} if $\varphi$ is a conjugate, with respect to the conjugation action given by $\mathrm{PGL}_2(\K)$ on $\K(X)$, 
either to $\pm X^d$ or to $\pm T_d(X)$. 
}
\end{definition}

Here, we use $T_d$ to denote the Chebyshev polynomial  of degree $d$ which is uniquely defined
 by the functional equation $T_d(X+X^{-1})=X^d+X^{-d}$.

Throughout the paper, we use the Landau symbol $O$ and the Vinogradov symbol $\ll$. Recall that the
assertions $U=O(V)$ and $U \ll V$  are both equivalent to the inequality $|U|\le cV$ with some absolute constant $c>0$.
  To emphasise the dependence of the implied
constant $c$ on some parameter (or a list of parameters) $\rho$, we write $U=O_{\rho}(V)$
or $U \ll_{\rho} V$.

\subsection{Main results and methods} 
\label{sec:res}

First, we show that under rather natural conditions 
on $\bphi \in \K(X)^s$ the point $\bphi(\alpha)$ 
is multiplicatively independent for all but finitely many elements
$\alpha \in \Kab$, see Theorem~\ref{thm:mult-Kc} below.
 
Then, in Section~\ref{sec:mult dep orb} we establish several 
results about multiplicative independence of $s\ge 1$ consecutive 
elements in an orbit of a polynomial  for all but finitely many 
initial values $\alpha \in \Kc$. These give a very broad generalisation 
of previously known results on roots of unity in orbits (which corresponds to
$s=1$). 

In Section~\ref{sec:mult dep arb} we investigate the question of 
multiplicative dependence of pairs of elements (not necessary consecutive) in orbits. 
More precisely, we prove in Theorem~\ref{thm: 2multdep-arb} that 
for a polynomial $f\in\K[X]$ of degree at least two and without multiple roots, 
 there are only finitely many elements $\alpha \in \K$ such that for some  distinct integers $ m, n\ge 0$
the values $f^{(m)}(\alpha)$ and $f^{(n)}(\alpha)$ are multiplicatively dependent. 
In particular, this can be considered as an 
extension of the {\it Northcott theorem\/}, see~\cite[Theorem~3.12]{Silv}, 
about the finiteness of preperiodic points (clearly $f^{(m)}(\alpha)  = f^{(n)}(\alpha)$ can be considered 
as a very special instance of multiplicative dependence).

The proofs of the above results rest  on Section~\ref{sec:prelim}, where we collect 
several general statements about polynomials and their iterations, 
and on Section~\ref{sec:arith prop} where we present more specialised  auxiliary results,
some of  which are new and maybe of independent interest.

The results of Section~\ref{sec:mult dep orb} are based on a combination of several ideas. 
First we need to record a  rather precise description of the structure of multiplicatively dependent 
values of rational functions in elements from $\Kab$,  see Lemma~\ref{lem:mult-Kab} below, complementing~\cite[Theorem~2.1]{OSSZ}.
The proof is based on the ideas of Bombieri, Masser and Zannier~\cite{BMZ} and
is similar to those in~\cite{OSSZ}.  This description is then combined with  an argument of 
Dvornicich and Zannier~\cite{DZ}, which has recently also been used in~\cite{Ost}. 
However, compared to the orginial scheme of~\cite{DZ} now somewhat simpler
argument is possible, thanks to the results of Fuchs and Zannier~\cite{FZ}, which in turn 
extend previous results of  Zannier~\cite{Zan0} from Laurent polynomials to arbitrary 
rational functions.  

The main result of Section~\ref{sec:mult dep arb}, that is, Theorem~\ref{thm: 2multdep-arb}
is based on some classical Diophantine techniques. More precisely, we use results about 
the finiteness of perfect powers amongst  polynomial values, which are due to B{\'e}rczes,   Evertse and Gy{\" o}ry~\cite{BEG}, combined with the celebrated result of Faltings~\cite{Fa83} on the finiteness of rational points on a plane curve of genus $g>1$.

\subsection{Further generalisations}

It is easy to see that some of our results can be extended to any field  with  
the Bogomolov property, that is, 
fields $\L \subseteq \ovQ$ for which   there exists a  constant $c_\L> 0$, such that for any 
non-zero $\alpha \in \L\setminus \U$ we 
have $\h(\alpha) \ge c_\L$. In particular, from~\cite[Theorem~1.2]{AmZa} it follows that  $\Kab$ has
the Bogomolov property, see~\cite{ADZ,Gal,Giz,Hab} for non-abelian examples of such fields and
further references.

\section{Preliminaries}
\label{sec:prelim}

\subsection{Bounding exponents in multiplicative relations} 

One of our main tools is the following result obtained in the proof of~\cite[Theorem~2.1]{OSSZ}, which follows the same approach as in~\cite[Theorem~1]{BMZ}.
It can be seen as a generalisation of Loxton and van der Poorten's result~\cite{Loxton, Poorten} on bounding exponents of multiplicative relations of algebraic numbers. 

\begin{lemma}
\label{lem:exponent}
Let $\bphi=(\varphi_1,\ldots,\varphi_s) \in \K(X)^s$ whose components are multiplicatively independent modulo constants. 
Then, for any $\alpha \in \Kab$ such that the point 
$\bphi(\alpha)$ is  multiplicatively dependent, there exist integers $k_1,\ldots,k_s$, not all zero, satisfying
\begin{equation*} 
\max|k_i| \ll_{d_\K,\bphi} 1,
\end{equation*}
and such that
\begin{equation*}
\varphi_1(\alpha)^{k_1}\cdots \varphi_s(\alpha)^{k_s}=\zeta 
\end{equation*}
for some root of unity $\zeta\in\U\cap \K(\alpha)$.
\end{lemma}

We remark that the necessary condition ``multiplicatively independent modulo constants'' in Lemma~\ref{lem:exponent} comes from~\cite[Theorem~1 and Theorem~1']{BMZ}; 
see also~\cite[Remark~2.6]{OSSZ}. 
Here, we present a simple example. 
Take $\varphi_1=X+1, \varphi_2= X-1, \varphi_3=2(X^2-1)$, then $\varphi_1, \varphi_2, \varphi_3$ are multiplicatively independent 
but they are not multiplicatively independent modulo constants. 
For any integer $m \ge 2$, let $\alpha_m = 2^m-1$. Then, we have the multiplicative dependence relation
$$
\varphi_1(\alpha_m)^{-(m+1)} \varphi_2(\alpha_m)^{-m} \varphi_3(\alpha_m)^m = 1, 
$$
where the absolute values of the exponents go to  infinity as $m$ goes to  infinity.

\subsection{Hilbert's Irreducibility Theorem over $\Kc$} 
We need the following result due to Dvornicich and Zannier~\cite[Corollary~1]{DZ}. We present it however in a weaker form as in~\cite[Lemma~2.1]{Ost} that is needed for our purpose, but the proof is given within the proof of~\cite[Corollary~1]{DZ}.

\begin{lemma}
\label{lem:cor1DZ}
Let $f\in \Kc[X,Y]$ be such that $f(X,Y^m)$ as a polynomial in X does not have a root in $\Kc(Y)$ for all positive integers $m\le \deg_X f$. 
Then, $f(X,\zeta)$ has a root in $\Kc$ for only finitely many roots of unity $\zeta$.
\end{lemma}

\subsection{Representations via linear combinations of roots of unity}
Loxton~\cite[Theorem~1]{Loxt} has proved that any algebraic integer $\alpha$ contained in some cyclotomic field has a short representation as a sum of roots of unity, 
that is, $\alpha=\sum_{i=1}^b\zeta_i$, where $\zeta_1,\ldots,\zeta_b\in\U$, and the integer $b$ depends only on $\house{\strut\alpha}$. 

Dvornicich and Zannier~\cite[Theorem~L]{DZ} extended the result of Loxton~\cite[Theorem~1]{Loxt} to algebraic integers 
contained in a cyclotomic extension of a given number field. Here we present a simplified version. 

\begin{lemma}
\label{lem:loxton}
There exists a finite set $\cE\subseteq \K$ depending on $\K$  such that any algebraic integer $\alpha\in \Kc$ 
can be written as $\alpha=\sum_{i=1}^bc_i\xi_i$, where $c_i\in \cE$, $\xi_i\in\U$, $i =1, \ldots, b$,  and the integer $b$ depends only on $\K$ and $\house{\strut\alpha}$.
\end{lemma}

\subsection{Multiplicative independence of polynomial iterates} 

We need the following special case of the result of Young~\cite[Corollary~1.2]{You}, which 
generalises the previous result of Gao~\cite[Theorem~1.4]{Gao}  to multiplicative independence of 
consecutive iterations of polynomials over fields of characteristic zero. 

\begin{lemma}
\label{lem:Gao} Let $\F$ be an arbitrary field of characteristic zero, and let 
$f \in \F[X]$ be a polynomial of degree at least $2$ which is not a monomial. 
Then, for any fixed integer $n\ge 1$, the  polynomials 
$f^{(1)}(X), \ldots,f^{(n)}(X)$ are multiplicatively independent modulo constants.
\end{lemma}

We also note that the result of~\cite[Corollary~1.2]{You} applies to rational functions as well, under some mild conditions.

\section{Arithmetic properties of polynomials and their iterations}
\label{sec:arith prop}

\subsection{Growth of the number of terms in  iterates of rational functions} 

Another important tool for our main results is a bound of 
Fuchs and Zannier~\cite[Corollary]{FZ}, 
on the number of terms in the iterates of a rational function.  First we introduce the following:
 
 \begin{definition} [\textit{Sparsity of rational functions}]
\label{def:Spr}
{\rm 
We define the \textit{sparsity} $\Sp(\varphi)$ of a rational function $\varphi\in\C(X)$  as the smallest 
number of total terms (in both numerator and denominator) in any representation  $\varphi = f/g$ with $f,g\in\C[X]$. 
}
\end{definition}

It is important to note that in Definition~\ref{def:Spr} we do not impose the coprimality 
condition on the polynomials $f$ and $g$.
 
 We present next the  result of 
Fuchs and Zannier~\cite[Corollary]{FZ} in the form that is needed for this paper, that is, for iterates of polynomials. We note that their result is proven for iterates of rational functions of degree $d\ge 3$. 
Although it is very likely that~\cite[Corollary]{FZ} can be extended to rational functions of degree $d = 2$  as well, 
see the comments after~\cite[Corollary]{FZ}, here we choose a simpler path
which is quite sufficient for our purpose. First we need the following simple statement
about the   Chebyshev polynomial $T_4$  of degree $4$.

\begin{lemma}
\label{lem:T4T2}
Let $\F$ be a field of characteristic zero. 
If $T_4(X)  = f^{(2)}(X)$ for some polynomial $f \in \F[X]$, 
then $f(X) = T_2(X)$. 
\end{lemma}

\begin{proof}
Recall the explicit form $T_4(X) = X^4-4X^2+2$.
We can assume that $f(X) = aX^2 + bX +c, a \ne 0$. Clearly, the coefficient  
of $X^3$ in $f^{(2)}(X)$ comes only from the term 
$af(X)^2$ and is equal to $2 a^2 b$. Since $a \ne 0$, we have $b =0$.
Thus, $f(X) = aX^2 + c$ and
$f^{(2)}(X)= a^3X^4 + 2a^2cX^2 + ac^2 +c$.
Therefore, 
$$
a^3 = 1, \qquad  2a^2c = -4, \qquad ac^2 +c = 2.
$$
From the second relation above,  we derive $2a^3c = - 4a$, and since $a^ 3= 1$ we obtain 
$c = -2a$. Hence,
$$
ac^2 +c = 4a^3 - 2a  = 4 - 2a, 
$$
and recalling  the third  relation, we obtain $a=1$ and thus $c =-2$.
Therefore, $f(X) = X^2 - 2 = T_2(X)$. 
\end{proof}

Now, we are ready to present the following slight variation of the result of Fuchs and Zannier in~\cite[Corollary]{FZ}. 

\begin{lemma}
\label{lem:FZ-2}
Let $\F$ be a field of characteristic zero. 
Let $q\in\F(X)$ be a non-constant rational function, and let $f\in\F[X]$ be of degree $d \ge 2$. Assume that $f$ is not  special. Then,   for any $n\ge 1$ we 
have 
$$
\Sp\(f^{(n)}(q(X))\) \ge  \((n-5)\log d - \log 2016\)/\log 5.
$$
\end{lemma}

\begin{proof}
If $d \ge 3$, then the claimed result is included in~\cite[Corollary]{FZ}. 
In the following, we assume that $d=2$. 

Since $\Sp\(f^{(n)}(q(X))\) \ge 1$ we can obviously assume that $n>5$. 

If $n =2m$ is even, then setting $g(X)  = f^{(2)}(X)$, we see from 
 Lemma~\ref{lem:T4T2} that $g$ is also not special, and then for 
$f^{(n)}(q(X)) = g^{(m)}(q(X))$, by~\cite[Corollary]{FZ}  we have
\begin{align*}
\Sp\(f^{(n)}(q(X))\) & = \Sp\(g^{(m)}(q(X))\)  \\
& \ge \((m-2)\log 4 - \log 2016\)/\log 5\\
 & = \((n/2-2)\log 4 - \log 2016\)/\log 5\\
 & =  \((n-4) \log 2-  \log 2016\)/\log 5, 
\end{align*}
which is better than the claimed result. 

Now, assume that $n =2m-1$  is odd. 
Since $n >5$, we have $m>3$. 
We rewrite 
$$
\Sp\(f^{(n)}(q(X))\)=\Sp\(f^{(2(m-1))}(f(q(X)))\), 
$$
then as the previous case, we obtain 
\begin{align*}
\Sp\(f^{(n)}(q(X))\) & \ge \((m-3)\log 4 - \log 2016\)/\log 5\\
 & =  \((n-5) \log 2-  \log 2016\)/\log 5.
\end{align*}
This completes the proof. 
\end{proof}

\subsection{On the size of elements in  orbits}

We  make use of the following simple facts about the growth of elements in orbits, 
which are given in different forms or follow exactly the same proof as in~\cite[Lemma~3.5]{Chen} 
and~\cite[Lemma~2.5, Lemma~2.7 and Corollary~2.8]{Ost}.

We separate them into Archimedean and non-Archimedean cases. 

\begin{lemma}
\label{lem:itergrowth Arch}
Let $f(X)=\sum_{i=0}^da_iX^i \in\K[X]$ of degree $d \ge 2$ and
let 
\begin{equation}
\label{eq:L}
L=\max_{\sigma} \(1+|\sigma(a_d^{-1})|+\sum_{i=0}^{d-1}|\sigma(a_i/a_d)|\),
\end{equation}
where the maximum is taken over all embeddings $\sigma$ of $\K$ into $\C$. 
Then, if $\alpha\in \ov\Q$ is such that
\begin{itemize}
\item[(i)] 
$|\alpha|>L$, then $\{|\sigma(f)^{(n)}(\alpha)|\}_{n\in\N}$ is a strictly increasing sequence
for any embedding $\sigma$ of $\K$ into $\C$; 

\item[(ii)]  $\house{\strut f^{(n)}(\alpha)}\le A$ for some positive real number $A$ and some integer $n\ge 1$, then 
$$
\house{\strut f^{(r)}(\alpha)} \le \max (A,L),   \quad \textrm{for all $0\le r\le n-1$}.
$$
\end{itemize} 
\end{lemma}

\begin{lemma}
\label{lem:itergrowth nonArch}
Let $f(X)=\sum_{i=0}^da_iX^i \in\K[X]$ of degree $d \ge 2$ and let $\alpha\in \ov\Q$ be such that 
\begin{equation}
\label{eq:alphacondpoly}
|\alpha|_v>\max_{j=0,\ldots,d-1} \{1,|a_j/a_d|_v, |a_d|_v^{-1}\}
\end{equation} 
for a non-archimidean absolute value $|\cdot|_v$ of $\K$ (normalised in some way and extended to $\ov\Q$). 
Then, $\{|f^{(n)}(\alpha)|_v\}_{n\in\N}$ is a strictly increasing sequence. 
\end{lemma}

However, we also need a different form of Lemmas~\ref{lem:itergrowth Arch}
and~\ref{lem:itergrowth nonArch}, which we present below.

\begin{lemma}
\label{lem:itergrowthext}
Let $f(X)=\sum_{i=0}^da_iX^i \in\K[X]$ of degree $d\ge 2$. Then, there exists a real number $L>0$ and an integer $E$, both depending only on $f$, 
such that, for any non-zero $\alpha, \beta \in\ov\Q$   with $\house{\strut\beta}\ll1$,  we have:
\begin{itemize}
\item[(i)] If $\house{\strut\beta f^{(n)}(\alpha)} \le A$ for some constant $A$ and some $n\ge 1$, then 
$$
\house{\strut\beta f^{(r)}(\alpha)} \ll \max(A,L)
$$ 
for any $0\le r\le n-1$.
\item[(ii)] If $\beta$ is an algebraic integer and $\beta f^{(n)}(\alpha)$ is also an algebraic integer for some $n\ge 1$, then 
$E \beta f^{(r)}(\alpha)$ is also an algebraic integer  for any $0\le r\le n-1$.
\end{itemize} 
\end{lemma} 
\begin{proof}
(i) 
Let $g(X)=\beta f(\beta^{-1}X)$ and let $L$ be defined by~\eqref{eq:L}. 
Then, for any $k\ge 1$, one has 
$$
g^{(k)}(X)=\beta f^{(k)}(\beta^{-1}X).
$$
From the hypothesis, we know that $\house{\strut g^{(n)}(\beta\alpha)} \le A$ for some $n\ge 1$.
Now, applying Lemma~\ref{lem:itergrowth Arch} with the polynomial $g$ and the initial point $\beta\alpha$, 
we obtain that 
$$
\house{\strut \beta f^{(r)}(\alpha)}=\house{\strut g^{(r)}(\beta\alpha)} \le \max (A,L_g)
$$
 for all $0\le r\le n-1$, where  as in~\eqref{eq:L} 
$$
L_g=\max_{\sigma} \(1+|\sigma(\beta)|^{d-1}|\sigma(a_d^{-1})|+\sum_{i=0}^{d-1}|\sigma(\beta)|^{d-i}|\sigma(a_i/a_d)|\),
$$
where the maximum is taken over all embeddings $\sigma$ of $\K$ into $\C$.
However, since $\house{\strut \beta} \ll 1$, we have $L_g \ll L$, and thus 
$$
\house{\strut \beta f^{(r)}(\alpha)} \ll \max (A,L)
$$
 for all $0\le r\le n-1$. This concludes the part~(i).

(ii) 
We enlarge the field $\K$ if needed such that it contains $\alpha, \beta$. 
Let $|\cdot|_v$ be any non-archimidean absolute value  of $\K$.   
From the hypothesis we know that $g^{(n)}(\beta\alpha)$ is an algebraic integer, and so  $|g^{(n)}(\beta\alpha)|_v\le 1$. 
Then, by Lemma~\ref{lem:itergrowth nonArch}, applied again to the polynomial $g$ and the initial points $g^{(r)}(\beta\alpha), r=0,1,\ldots, n-1$, 
we see that the analogue of the condition~\eqref{eq:alphacondpoly} fails  and we obtain that
$$
|g^{(r)}(\beta\alpha)|_v \le \max_{i=0,\ldots,d-1} \{1,|\beta|_v^{d-i}|a_i/a_d|_v,|\beta|_v^{d-1}|a_d|_v^{-1}\}.
$$
Since $|\beta|_v\le 1$ (note that $\beta$ is an algebraic integer), we deduce that
$$
|\beta f^{(r)}(\alpha)|_v=|g^{(r)}(\beta\alpha)|_v\le \max_{i=0,\ldots,d-1} \{1,|a_i/a_d|_v,|a_{d}|_v^{-1}\}.
$$
Taking now a sufficiently large integer $E$ such that $Ea_i/a_d$ and $E/a_d$ are algebraic integers and so $|Ea_i/a_d|_v, |Ea_d^{-1}|_v\le 1$, for $i=0,\ldots,d-1$, 
we have $|E\beta f^{(r)}(\alpha)|_v \le 1$ and conclude that $E\beta f^{(r)}(\alpha)$ is an algebraic integer for any integer  $r$ with $0\le r\le n-1$.
\end{proof}

\subsection{Compositions of rational functions and monomials} 

We need the following result which claims that the compositions of rational functions 
usually  cannot give a monomial.

\begin{lemma}
\label{lem:DZcond}
Let $\varphi \in\overline{\Q}(X)$ be a rational function.  
Assume that $\varphi$ is not a power of a linear fractional function.
Then,  for 
$$
R(X,Y)=\varphi(X) - bY, \quad b \in \ovQ^*, 
$$
there exists no rational function $S\in\overline{\Q}(Y)$ such that
$$
R(S(Y),Y^m)=0 \quad \textrm{ for some $m\ge 1$}.
$$
\end{lemma}

\begin{proof}  
Without loss of generality, we can assume that $\varphi$ is non-constant. 
By contradiction, suppose that there exists a non-constant rational function $S(Y)=S_1(Y)/S_2(Y)\in\overline{\Q}(Y)$, where $S_1,S_2\in\overline{\Q}[X]$ with $\gcd(S_1,S_2)=1$, and a positive integer $m\ge 1$ such that
\begin{equation}
\label{eq:RSY}
R(S(Y),Y^m)=0.
\end{equation}
Let $\beta_{1},\ldots,\beta_t\in\ov\Q$ be all the distinct zeros and poles  of $\varphi$.  Then
$R(X,Y)$ becomes
$$
R(X,Y)=\frac{a\prod_{j: D_j> 0}(X-\beta_j)^{D_j}}{\prod_{j: D_j<0}(X-\beta_j)^{-D_j}}- bY,
$$
for some $a \in \overline{\Q}^*$ and some non-zero integers  $D_1, \ldots, D_t$. 
If $\varphi$ has no poles, then the factor $\prod_{j: D_j<0}(X-\beta_j)^{-D_j}$ automatically becomes $1$. 
We set a similar convention when $\varphi$ has no zeros. 

We denote 
$$
D = D_1+\ldots + D_t.
$$
Clearing the denominators, we can rewrite~\eqref{eq:RSY} as
\begin{equation}
\label{eq:polyid}
\begin{split}
&a\prod_{j: D_j >0}(S_1(Y)-\beta_jS_2(Y))^{D_j}\\
&\qquad = bY^m S_2(Y)^{D} \prod_{j: D_j<0}(S_1(Y)-\beta_jS_2(Y))^{-D_j}.
\end{split}
\end{equation}

Since $\gcd(S_1,S_2)=1$ and all $\beta_i$, $i=1,\ldots,t$, are pairwise distinct, we know that for any $i\ne j$ we have
\begin{equation}
\label{eq:gcd cond}
\begin{split}
&\gcd(S_1(Y)-\beta_iS_2(Y),S_1(Y)-\beta_jS_2(Y))=1,\\ &\gcd(S_2(Y),S_1(Y)-\beta_jS_2(Y))=1.
\end{split}
\end{equation} 

We remark that if $S_2$ is constant then the factors $S_1 - \beta_j S_2$ are not constant
and that, otherwise, at least one of $S_1- \beta_iS_2, S_1-\beta_jS_2$ is not constant 
for $i \ne j$. Then, since $Y^mS_2^D$ is not constant, it is easy to see from~\eqref{eq:polyid}, \eqref{eq:gcd cond} and from
this remark, that  if $D \ne 0$ then $S_2$ is
constant and $\varphi$ has only one zero of multiplicity $m$. 
Similarly, if $D=0$ then $\varphi$ has exactly one zero and one pole both of multiplicity $m$. 
This  completes the proof by noticing the choice of $\varphi$. 
\end{proof}

\begin{remark} 
{\rm
The assumption in Lemma~\ref{lem:DZcond} is necessary. 
For example, if $\varphi(X)=(aX+b)^m/(cX+d)^m$ with $ad-bc \ne 0$ and $m\ge 1$, then one can take $S(Y)=(-dY+b)/(cY-a)$
 to conclude that $\varphi(S(Y))-Y^m=0$.
}
\end{remark}

\subsection{Structure of multiplicatively dependent values}

We start by supplementing the result of Bombieri, Masser and Zannier~\cite[Theorem~1]{BMZ} and
also its generalisation in~\cite[Theorem~2.1]{OSSZ} with a more explicit description of
 the multiplicatively dependent values of rational functions.

\begin{lemma}
\label{lem:mult-Kab}
Let $\bphi=(\varphi_1, \ldots, \varphi_s) \in\K(X)^s$ whose components are multiplicatively independent modulo constants. 
Then, there exist an integer $A\ge 1$ and a finite set $S \subseteq \K^*$ both 
depending only on $d_\K$ and $\bphi$, such that 
each element $\alpha\in\Kab$, for which $\bphi(\alpha)$  is a multiplicatively dependent point, 
satisfies $[\Kc(\alpha):\Kc] \le A$ and
$$
\alpha=\frac{\gamma}{a-\eta},
$$
where $a \in S$, $\eta\in\U\cap \K(\alpha)$, 
and $\gamma\in\K(\alpha)$ with $\house{\strut \gamma} \le A$ 
and   $A\gamma\in \Z_{\K(\alpha)}$. 
In particular, if $\varphi_1, \ldots, \varphi_s$ are all monic, one can choose $a=1$. 
\end{lemma}

\begin{proof} 
Let $\alpha\in\Kab$ be such that $\bphi(\alpha)$ is multiplicatively dependent. 
Then, by Lemma~\ref{lem:exponent}, we know that there exist $k_1,\ldots,k_s\in\Z$, not all zero, satisfying
\begin{equation}  \label{eq:ki}
\max|k_i| \ll_{d_\K,\bphi} 1,
\end{equation}
and such that
\begin{equation}
\label{eq:multdepU}
\varphi_1(\alpha)^{k_1}\cdots \varphi_s(\alpha)^{k_s}=\zeta
\end{equation}
for some root of unity $\zeta\in\U\cap \K(\alpha)$.

For each $i=1,\ldots,s$, write $\varphi_i = f_{i} / g_{i} $, 
where $ f_{i},g_{i} \in \K[X]$ with $\gcd( f_{i},g_{i} )=1$. 
 Then, from~\eqref{eq:multdepU} we obtain that $\alpha$ is a root of the polynomial
\begin{equation}   \label{eq:Psi}
\begin{split}
\Psi(X) = \prod_{\substack{i=1\\ k_i > 0}}^s  f_{i}(X)^{k_i} & \prod_{\substack{i=1\\ k_i < 0}}^s g_{i} (X)^{-k_i} \\
 & - \zeta \prod_{\substack{i=1\\ k_i > 0}}^s g_{i} (X)^{k_i} \prod_{\substack{i=1\\ k_i < 0}}^s  f_{i}(X)^{-k_i}
 \end{split}
\end{equation}
 with coefficients of absolute value upper bounded only in terms of $d_{\K}$ and $\bphi$.  
 Since $\varphi_1, \ldots, \varphi_s$ are multiplicatively independent, $\Psi$ is a non-zero polynomial. 
 
In view of~\eqref{eq:Psi}, 
we can find a positive integer $D$ depending only on $d_{\K},\bphi$ such that $\deg \Psi \le D$, which implies that 
$$
[\Kc(\alpha):\Kc]  \le D.
$$

Note that we have two cases. One is that  the leading 
coefficient of $\Psi(X)$ is $b$ or $b\zeta$ 
for some $b\in \K^*$ ($b$ depends only on $d_{\K},\bphi$), and so in this case, 
it is easy to see that $E_1\alpha$ is an algebraic integer for some large integer $E_1$ depending only on $d_{\K},\bphi$, and 
$$
\house{\strut \alpha} \ll_{d_{\K},\bphi} 1. 
$$ 
Thus, the claimed form of $\alpha$ follows by choosing $a=1, \eta = -1$ and putting $\gamma=\alpha(1-\eta)=2\alpha$, where we still have 
\begin{equation}   \label{eq:gamma1}
\house{\strut \gamma} \ll_{d_{\K},\bphi} 1 \quad \textrm{and} \quad \textrm{$E_1\gamma$ is an algebraic integer}.
\end{equation}
The other case is that  the leading coefficient of $\Psi(X)$ is $c-b\zeta$ for some $b,c\in \K^*$, where $b,c$ depend only on $d_{\K},\bphi$. 
In particular, if $\varphi_1,\ldots,\varphi_s$ are all monic, we have $b=c=1$. 
Extending the products in $\Psi(X)$, we obtain that $\alpha$ satisfies an equation of the form  
$$
(c-b\zeta) \alpha^e + \sum_{i=0}^{e-1} b_i \alpha^i = 0,
$$
or equivalently, if we denote $\beta = \alpha (c-b\zeta)$,
\begin{equation*}
\label{eq:alpha}
\beta^e + \sum_{i=0}^{e-1} (c-b\zeta)^{e-1-i}b_i \beta^i = 0,
\end{equation*}
for some integer $e\ge 1$ depending only on $d_{\K},\bphi$ and with coefficients $b_i$, $i=0,1,\ldots,e-1$, with 
\begin{equation}
\label{eq:coeff cond}
|b_i|\ll_{d_{\K},\bphi} 1 \quad \textrm{and} \quad |(c-b\zeta)^{e-1-i}| \ll_{d_{\K},\bphi} 1.  
 \end{equation}
So, it is easy to see that we can choose a large integer $E_2$ depending only on $d_{\K},\bphi$ 
such that $E_2\beta$ is an algebraic integer, and by~\eqref{eq:coeff cond} we also have 
$$
\house{\strut \beta}\ll_{d_{\K},\bphi} 1. 
$$

From $\beta = \alpha (c-b\zeta)$, we have 
$$
\alpha = \frac{\beta / b}{c/b - \zeta}. 
$$
Then, let $a=c/b, \gamma = \beta / b$ and $\eta=\zeta$. 
Based on the choices of $b,c$, we know that $a \in \K^*$ depends only on $d_\K,\bphi$, 
and thus the finiteness of choices of $a$ follows from~\eqref{eq:ki}. 
We can also  enlarge $E_2$ if needed such that $E_2\gamma$ is also an algebraic integer ($E_2$ still depends only on $d_{\K},\bphi$), and also 
\begin{equation}   \label{eq:gamma2}
\house{\strut \gamma} \ll_{d_{\K},\bphi} 1. 
\end{equation}

Thus, from~\eqref{eq:gamma1} and~\eqref{eq:gamma2} we have proved that there exists an integer $B\ge 1$ that depends only on $d_{\K},\bphi$ such that  
we always have $\house{\strut \gamma} \le B$.

Taking now $A=\max(B,D,E_1,E_2)$, we conclude the proof. 
\end{proof}

\begin{remark}
{\rm 
We see from Lemma~\ref{lem:mult-Kab} that if the components of $\bphi$ are all monic, 
then for each such $\alpha \in \Kab$, there exists an algebraic integer $\delta \in \Kab$ (that is, $\delta=A(1-\eta)$ with $A$ and $\eta$ as in Lemma~\ref{lem:mult-Kab})
with $|\delta| \le 2A$ such that $\alpha\delta$ is an algebraic integer, 
which roughly means that the ``denominator'' of $\alpha$ is uniformly bounded. 
}
\end{remark}

\begin{remark}
{\rm 
In Lemma~\ref{lem:mult-Kab}, even if we assume that such $\alpha$ are algebraic integers, we cannot have that the house $\house{\strut \alpha}$ of such $\alpha$ is 
uniformly upper bounded (that is, only in terms of $d_{\K},\bphi$); see Example~\ref{ex:house} below. 
}
\end{remark}

\begin{example}   \label{ex:house}
{\rm
Choose $\varphi_1(X)=X$ and $\varphi_2(X)=X+1$. Certainly, $\varphi_1$ and $\varphi_2$ are multiplicatively independent modulo constants. 
However, for any $n$-th root of unity $\zeta_n \ne 1$, let $\alpha=1/(\zeta_n-1)$. It is easy to see that $\varphi_1(\alpha)$ and 
$\varphi_2(\alpha)$ are multiplicatively dependent. In addition, it is well-known (see~\cite[Proposition~2.8]{Wash}) 
that if $n$ has at least two distinct prime factors, then $\zeta_n-1$ is a unit of $\Z[\zeta_n]$, 
and thus $1/(\zeta_n-1)$ is an algebraic integer. 
}
\end{example}

\begin{remark}  \label{rem:finite}
{\rm 
There are also some cases where one can claim the finiteness of the set of $\alpha$ in Lemma~\ref{lem:mult-Kab}. 
Here are some examples.
\begin{enumerate}
\item If $\bphi\in (\K(X) \cap \R(X))^s$, then one  immediately obtains the  finiteness of such $\alpha \in \Kab\cap \R$ in Lemma~\ref{lem:mult-Kab}. 
This follows directly from the proof of Lemma~\ref{lem:mult-Kab}, since in this case in~\eqref{eq:multdepU} we have $\zeta=\pm 1$ 
because the left-hand side is real, and then in~\eqref{eq:Psi} the polynomial $\Psi(X)$ is defined over $\K$ and has bounded degree. 
So, such elements $\alpha \in \Kab\cap \R$ are of bounded degree and bounded height, which leads to the finiteness. 

\item Similarly, if $\K$ is a totally real number field,  
then there are only finitely many elements $\alpha\in \Qtr$, where $\Qtr$ is the field of all totally real algebraic numbers,   such that the point
$\bphi(\alpha)$ is multiplicatively dependent. This conclusion follows since Lemma~\ref{lem:exponent} still holds when we replace $\Kab$ by $\Qtr$. Indeed, the proof of Lemma~\ref{lem:exponent} for $\Qtr$ follows the same as in the proof of~\cite[Theorem~2.1]{OSSZ}, where instead of~\cite[Theorem~1.2]{AmZa} one uses an early result due to Schinzel~\cite{Schin} on the heights of totally real algebraic numbers. 
 So, in this case, the inequality~\eqref{eq:ki} still holds, and in~\eqref{eq:multdepU} we also have $\zeta=\pm 1$. 
\end{enumerate}
}
\end{remark}

For the two examples in Remark~\ref{rem:finite}, we want to indicate that $\Qab \cap \R \subsetneq \Qtr$. 
This is based on two facts: one is that any abelian extension of $\Q$ is either totally real, or contains a totally real subfield over which it has degree two; 
and the other is that there exist totally real fields which are not abelian over $\Q$.

\section{Main Results}
\label{sec:main}

\subsection{Finiteness of multiplicatively dependent values}
\label{sec:mult dep val}

Now, we  give a stronger version of Lemma~\ref{lem:mult-Kab} by proving finiteness of $\alpha\in\Kab$ 
for which $\bphi(\alpha)$ is a multiplicatively dependent point for a class of 
 $\bphi\in\K(X)^s$. First, we introduce a definition.

\begin{definition}
{\rm 
We say that the rational functions $\varphi_1,\ldots,\varphi_s\in\K(X)$ \textit{multiplicatively generate} a power of a linear fractional function 
if  there exists integers $k_1,\ldots,k_s$, not all zero, such that $\varphi_1^{k_1}\cdots \varphi_s^{k_s}$ is a power of a linear fractional function. 
}
\end{definition}

Note that a linear fractional function can be a  constant function, and  the zero power of a linear fractional function is $1$ by convention. 
So, if $\varphi_1,\ldots,\varphi_s$ cannot multiplicatively generate a power of a linear fractional function, 
then they are automatically multiplicatively independent modulo constants.

The possibility of the following result  has been indicated in~\cite[Remark~4.2]{OSSZ}, 
here we   present it in full detail.

\begin{theorem}
\label{thm:mult-Kc}
Let $\bphi=(\varphi_1,\ldots,\varphi_s)\in\K(X)^s$ whose components cannot multiplicatively generate a power of a linear fractional function. 
Then, there are only finitely many elements
 $\alpha \in \Kab$ such that $\bphi(\alpha)$  is a multiplicatively dependent point. 
\end{theorem}

\begin{proof}
First, we remark that it is enough to prove the finiteness of the elements $\alpha\in\Kc$ such that $ \bphi(\alpha)$ is multiplicatively dependent. 
Indeed,  if $\cX$ is the rational curve parametrized by $\bphi$, then proving the theorem is
 equivalent to showing that there are only finitely many multiplicatively dependent points in $\cX(\Kab)$. 
By~\cite[Theorem~2.1 and Remark~2.2]{OSSZ}, the set of  dependent points in $\cX(\Kab)$ is the union of a finite set with the dependent ones in $\cX(\Kc)$. 

Let now $\alpha\in\Kc$ such that $\bphi(\alpha)$ is a multiplicatively dependent point. 
By Lemma~\ref{lem:exponent}, one can find integers $k_1,\ldots,k_s$, not all zero, that are uniformly bounded only in terms of $d_\K$ and $\bphi$, such that 
\begin{equation}   \label{eq:zeta}
\varphi_1(\alpha)^{k_1} \cdots\varphi_s(\alpha)^{k_s}=\zeta,
\end{equation}
for some root of unity $\zeta\in\U$.

Let 
$$
R(X,Y)=\varphi_1(X)^{k_1} \cdots\varphi_s(X)^{k_s}-Y.
$$
By assumption, $\varphi_1(X)^{k_1} \cdots\varphi_s(X)^{k_s}$ is not a power of a linear fractional function. 
Then, from Lemma~\ref{lem:DZcond}, we conclude that there is no rational function $S(Y)\in\Kc(Y)$ such that $R(S(Y),Y^m)=0$ for any $m\ge 1$. 
Applying now Lemma~\ref{lem:cor1DZ} to the numerator of $R(X,Y)$, we obtain that there are only finitely many roots of unity $\zeta\in\U$ such that $R(X,\zeta)$ has a zero in $\Kc$. 
Noticing again that $k_1,\ldots,k_s$ are all uniformly bounded only in terms of $d_\K$ and $\bphi$, 
we have that there are only finitely many equations of the form~\eqref{eq:zeta}, and thus we conclude the proof.
\end{proof}

\begin{remark} 
{\rm 
We remark that in the above finiteness result, the number of such $\alpha$ depends on $\K,\bphi$. 
Besides, Example~\ref{ex:house} suggests that the assumption in Theorem~\ref{thm:mult-Kc} is indeed necessary. 
}
\end{remark}

\begin{remark}  
{\rm 
Theorem~\ref{thm:mult-Kc} can be interpreted that under a rather weak condition, the rational curve $\cX$ parametrized by 
$\bphi$ has only finitely many multiplicatively dependent points defined over $\Kab$. 
}
\end{remark}

\subsection{Finiteness of consecutive multiplicatively dependent elements in orbits}
\label{sec:mult dep orb}

First we consider compositions of polynomial iterations with several multiplicatively  independent 
rational functions. These results generalise that of~\cite{Ost} on roots of unity in 
orbits of a wide class of rational functions.
For this we recall the definition of a special rational function as in Definition~\ref{def:Srf}.

\begin{theorem}
\label{thm:itera-infi}
Let $\bphi\in\K(X)^s$ whose components are multiplicatively  independent modulo constants, 
and let $f\in\K[X]$ be a non-special polynomial of degree at least $2$. 
Then, the following hold: 
\begin{itemize}
\item[(i)] there exists a non-negative integer $n_0$ depending only on $f,\bphi$ and $\K$ such that there are at most finitely 
many elements $\alpha \in \Kc$  for which
$\bphi\(f^{(n)}(\alpha)\)$ is a multiplicatively dependent point for some integer $n \ge n_0$; 

\item[(ii)] if furthermore the components of $\bphi$ cannot multiplicatively generate a power of a linear fractional function, 
then {\rm (i)} holds with $n_0=0$. 
\end{itemize}
\end{theorem}

\begin{proof} 
(i) 
Let $\alpha\in\Kc$ be such that  $\bphi\(f^{(n)}(\alpha)\)$ is   multiplicatively dependent for some integer $n \ge 0$. 
Then, applying Lemma~\ref{lem:mult-Kab} we have 
$$
f^{(n)}(\alpha)=\frac{\gamma_{\alpha,n}}{a_{\alpha,n}-\eta_{\alpha,n}},
$$
where $a_{\alpha,n}$ is in a finite set $S \subseteq \K^*$, $\eta_{\alpha,n}\in\U \cap \K(\alpha)$, and $\gamma_{\alpha,n}\in\K(\alpha)$ such that 
\begin{equation}
\label{eq:house n}
\house{\strut (a_{\alpha,n}-\eta_{\alpha,n})f^{(n)}(\alpha)} \ll_{d_{\K},\bphi} 1, 
\end{equation}
where $d_{\K}=[\K:\Q]$, and there exists a sufficiently large integer 
$$
A\ll_{d_{\K},\bphi} 1
$$ 
such that $A(a_{\alpha,n}-\eta_{\alpha,n})f^{(n)}(\alpha)$ is an algebraic integer.

We apply now the method of~\cite[Theorem~2]{DZ} and~\cite[Theorem~1.2]{Ost}.  Indeed, let $M$
be  a sufficiently large positive integer, chosen to satisfy
\begin{equation}
\label{eq:M large}
M>\frac{(B+2)\log 5+\log 2016}{\log d}+5,
\end{equation}
where $B$  defined below is a constant depending only on $f,\bphi$ and $\K$. 

We also denote by $\cS_{f,\bphi}(M)$ the set of $\alpha\in\Kc$ such that the point  
$\bphi\(f^{(n)}(\alpha)\)$ is multiplicatively dependent for some integer $n>M$. 
Our purpose is to show that $\cS_{f,\bphi}(M)$ is a finite set. 

Let now $\alpha\in\cS_{f,\bphi}(M)$. 
Using~\eqref{eq:house n}, we apply Lemma~\ref{lem:itergrowthext} to conclude, since $\house{\strut a_{\alpha,n}-\eta_{\alpha,n}} \ll_{d_{\K},\bphi} 1$ 
and $A(a_{\alpha,n}-\eta_{\alpha,n})f^{(n)}(\alpha)$ is an algebraic integer, that
$$
\house{\strut (a_{\alpha,n}-\eta_{\alpha,n})f^{(r)}(\alpha)} \ll_{f,d_{\K},\bphi} 1,\quad 0\le r\le M,
$$
and there exists a positive integer $E\ll_{f,d_{\K},\bphi} 1$ with the property that $E(a_{\alpha,n}-\eta_{\alpha,n})f^{(r)}(\alpha)$ is an algebraic integer 
for any $0\le r\le M$.

 Applying now Lemma~\ref{lem:loxton} for $E(a_{\alpha,n}-\eta_{\alpha,n})f^{(r)}(\alpha)$, $r=0,1,\ldots,M$, there exist a positive integer 
$B$ (depending only on $f,\bphi,\K$) and a finite set 
$\cE$ (depending only on  $\K$)
such that we can write $E(a_{\alpha,n}-\eta_{\alpha,n})f^{(r)}(\alpha)$ in the form 
\begin{equation}
\label{eq:iterlox}
E(a_{\alpha,n}-\eta_{\alpha,n})f^{(r)}(\alpha)=c_{\alpha,r,1}\xi_{\alpha, r,1}+\cdots+c_{\alpha,r,B}\xi_{\alpha, r,B}, 
\end{equation}
where $c_{\alpha,r,i}\in \cE$ and $\xi_{\alpha, r,i}\in\U$.

By contradiction, suppose that $\cS_{f,\bphi}(M)$ is an infinite set. 
Then, since both $S$ and $\cE$ are finite sets, we can choose an infinite subset $\cT_{f,\bphi}(M)$ of $\cS_{f,\bphi}(M)$ 
such that for any $\alpha \in \cT_{f,\bphi}(M)$, the coefficients $a_{\alpha,n}$ and $c_{\alpha,r,i} \in \cE$ in~\eqref{eq:iterlox} are all fixed (independent of $\alpha$) 
for $r=0,1,\ldots,M$ and $i=1,2,\ldots,B$.  
For these fixed coefficients, to simplify the notation from now on, we denote
$$
a=a_{\alpha,n} \quad \textrm{and} \quad c_{r,i}=c_{\alpha,r,i}.
$$ 
So, it suffices to consider the elements in $\cT_{f,\bphi}(M)$. 

We use the first equation corresponding to $r=0$ to replace $\alpha \in \cT_{f,\bphi}(M)$ on the left-hand side of~\eqref{eq:iterlox} and thus consider 
the equations with unknowns $Y, X_{r,i}, r=0,1,\ldots,M, i=1,2,\ldots,B$: 
\begin{equation}  \label{eq:itera-eq}
\begin{split}
E(a-Y)f^{(r)}&\Big(\frac{c_{0,1}X_{0,1}+\cdots+c_{0,B}X_{0,B}}{E(a-Y)}\Big)\\
&=c_{r,1}X_{r,1}+\cdots+c_{r,B}X_{r,B}, \qquad r=1,\ldots,M. 
\end{split}
\end{equation}
Then, for any $\alpha \in \cT_{f,\bphi}(M)$, the points  $(\eta_{\alpha,n},\xi_{\alpha, r,i})$, $r=0,1,\ldots,M$, $i=1,2,\ldots,B$, are torsion points on the 
variety $\cZ \subseteq \Gm^{B(M+1)+1}$ defined by the equations in~\eqref{eq:itera-eq}. 
Since $\cT_{f,\bphi}(M)$ is an infinite set and in view of~\eqref{eq:iterlox}, there are infinitely many such points. 

By the toric analogue of the Manin--Mumford conjecture  (also called the {\it torsion points theorem}) proved by Laurent~\cite{Laurent}  
and also more elementary by Sarnak and Adams~\cite{SA}, 
there exists a $1$-dimensional torsion coset of $\Gm^{B(M+1)+1}$ contained in the Zariski closure of the torsion points in the variety $\cZ$.
We can parametrize this coset by $X_{r,i}=\beta_{r,i}t^{e_{r,i}}$ and $Y=\tau t^{\ell}$ with a parameter $t$, where $\beta_{r,i},\tau$ are roots of unity and $e_{r,i},\ell$ are integers, not all zero, 
and so we obtain the following identities
\begin{equation}
\label{eq:ratiterlox}
f^{(r)}\(\frac{\sum_{j=1}^{B}c_{0,j}\beta_{0,j}t^{e_{0,j}}}{E(a-\tau t^{\ell})}\)=\frac{\sum_{j=1}^{B}c_{r,j}\beta_{r,j}t^{e_{r,j}}}{E(a-\tau t^{\ell})}, \qquad r=1,\ldots,M.
\end{equation}

Now, for 
$$
q(t)=\frac{\sum_{j=1}^{B}c_{0,j}\beta_{0,j}t^{e_{0,j}}}{E(a-\tau t^{\ell})},
$$
the equation~\eqref{eq:ratiterlox} shows that 
 $f^{(r)}(q(t))$, $r =1, \ldots, M$, is a   rational function  having all together 
 at most $B+2$  terms. 
Since $f$ is a non-special polynomial of degree at least $2$,  we directly apply Lemma~\ref{lem:FZ-2} to conclude that we must have
$$
B+2 \ge \frac{1}{\log 5}\((M-5)\log d - \log 2016\),
$$ 
which contradicts the choice of $M$ as in~\eqref{eq:M large}. 
Thus, $\cS_{f,\bphi}(M)$ is a finite set. 
Taking $n_0 = M+1$, this concludes the proof of the first part (i).

(ii) 
From Theorem~\ref{thm:mult-Kc}, we know that there are only finitely many elements $\beta\in\Kc$ such that $\bphi(\beta)$ is multiplicatively dependent. 
We can fix one such $\beta\in\Kc$, and we are thus left to prove that there are only finitely many $\alpha\in\Kc$ such that $f^{(n)}(\alpha)=\beta$ for some $n\ge 0$.

Let $M$ be a positive integer chosen to satisfy~\eqref{eq:M large}. 
It has been proved in the above that $\cS_{f,\bphi}(M)$ is a finite set. 
So, we only need to consider $n\le M$. 
Since $\beta\in\Kc$ is fixed,  
we conclude that there are only finitely many $\alpha\in\Kc$ satisfying $f^{(n)}(\alpha)=\beta$ for any $0\le n\le M$.
This  completes the proof. 
\end{proof}

\begin{remark}    \label{rem:general}
{\rm 
It is plausible that the result in Theorem~\ref{thm:itera-infi} and the rest of results of this section 
also hold when we replace the polynomial $f$ by a non-special rational function $g/h \in \K(X)$ with $g,h \in \K[X]$ and $\deg g - \deg h >1$; see~\cite{Chen, Ost}. 
}
\end{remark}

\begin{remark}
{\rm 
We note that a weaker statement of Theorem~\ref{thm:itera-infi} (ii) may also follow from~\cite[Theorems~1.5 and~2.5]{Chen}, 
under several other restrictions on the polynomial $f$. 
Indeed, from the proof of Theorem~\ref{thm:itera-infi} (ii), we can fix an element $\beta\in\Kc$ and reduce the problem to proving that there are only finitely many $\alpha\in\Kc$ such that $f^{(n)}(\alpha)=\beta$ for some $n\ge 1$.  In particular this means $\house{\strut f^{(n)}(\alpha)} \ll 1$.
We apply now~\cite[Theorem~1.5]{Chen} to conclude that there are only finitely many such $\alpha\in\Kc$, under the condition that 
 there does not exist a rational function $S(X)\in\Kc(X)$ such that $f(S(X))$ is a ``short" Laurent polynomial.
}
\end{remark}

\begin{remark}
{\rm 
From Theorem~\ref{thm:itera-infi} (ii), taking $s=1$ and $\bphi=f$, 
one can recover the main result in~\cite[Theorem~1.2]{Ost}: there are only finitely many $\alpha\in\K^c$ such that $f^{(n)}(\alpha)\in\U$ for some integer $n > 0$.
}
\end{remark}

Moreover, combining Theorem~\ref{thm:itera-infi} with Lemma~\ref{lem:Gao}, we have the following corollary.

\begin{corollary}
\label{cor:consecmultdepinf}
Let  $f \in\K[X]$ be non-special and of degree at least $2$. Then, for any integer $s \ge 1$, 
the following hold: 
\begin{itemize}
\item[(i)] there are only  finitely many $\alpha\in\Kc$ such that 
the $s$ consecutive iterations 
$f^{(n+1)}(\alpha),$ $\ldots,f^{(n+s)}(\alpha)$ are multiplicatively dependent for infinitely many integers $n \ge 0$; 

\item[(ii)] if furthermore the iterations of $f$ 
cannot   multiplicatively generate a power of a linear fractional function, 
then there are only finitely many $\alpha\in\Kc$ such that
the $s$ consecutive iterations 
$f^{(n+1)}(\alpha), \ldots,f^{(n+s)}(\alpha)$ are multiplicatively dependent for some integer $n\ge 0$.
\end{itemize}
\end{corollary}

\begin{remark} 
{\rm 
The condition of Corollary~\ref{cor:consecmultdepinf} (ii) on the iterations is not very 
restrictive. For example, it is easy to see that if $0$ is a not a periodic point of $f$, then all the iterations 
of $f$ are relatively prime. 
}
\end{remark}

\subsection{Finiteness of  multiplicatively dependent points in orbits at arbitrary positions }
\label{sec:mult dep arb}

Recall that $\ZK$ is the ring of integers of $\K$. 
Our main result below relies on some finiteness results on the number of perfect powers amongst polynomial  values (see~\cite{BEG}), as well as the result of Faltings~\cite{Fa83} on the finiteness of $\K$-rational points on curves of genus greater than one.

\begin{theorem}
\label{thm: 2multdep-arb}
Let $f\in\K[X]$ be a polynomial without multiple roots, of degree $d\ge 3$ or, if $d = 2$, we also assume that $f^{(2)}$ has no multiple roots.
Then, there are only finitely many elements $\alpha \in \K$ such that for some  distinct integers $ m, n\ge 0$
the values $f^{(m)}(\alpha)$ and $f^{(n)}(\alpha)$ are multiplicatively dependent.
\end{theorem}

\begin{proof}  
Let $d_{\K}=[\K:\Q]$, and write the $d_{\K}$ embeddings of $\K$ into $\C$ as $\sigma_1,\ldots,\sigma_{d_{\K}}$. 
For $i=1,\ldots,d_{\K}$, let   $f_i=\sigma_i(f)$.

Choose now a sufficiently large number 
$N > L$ (see the end of the proof), where $L$ is defined by~\eqref{eq:L}. 
Denote by $\cT(N,\K)$ the set  of elements $\alpha \in \K$ with $\house{\strut \alpha} \le N$. 
Then, $\cT(N,\K)$ is a set of bounded height contained in $\K$, and so it is  a finite set. 
Thus, we only need to consider elements in the set $\K \setminus \cT(N,\K)$. 

Now, for $\alpha \in \K \setminus \cT(N,\K)$, 
assume that there exists a pair of non-negative integers $(m,n)$ with $m>n$ such that
 for some integers $k$ and $\ell$ with $(k,\ell) \ne (0,0)$, we have 
\begin{equation}
\label{eq:mult kel}
\(f^{(m)}(\alpha)\)^k = \(f^{(n)}(\alpha)\)^\ell.
\end{equation} 
Note that there is an embedding, say $\sigma_j$, such that $|\sigma_j(\alpha)| = \house{\strut \alpha}$, 
and so $|\sigma_j(\alpha)| > N$.  
By~\eqref{eq:mult kel} we have 
$$
\(f_j^{(m)}(\sigma_j(\alpha))\)^k = \(f_j^{(n)}(\sigma_j(\alpha))\)^\ell.
$$
By the choices of $\sigma_j(\alpha),m$ and $n$, using Lemma~\ref{lem:itergrowth Arch} (applied with  $\sigma_j(\alpha)$ instead of $\alpha$) we obtain 
\begin{equation}    \label{eq:fj}
|f_j^{(m)}(\sigma_j(\alpha))| > |f_j^{(n)}(\sigma_j(\alpha))| \ge |\sigma_j(\alpha)| > N. 
\end{equation}
So, we cannot have $k=0$ or $\ell = 0$. 
Then, we must have 
$$
1\le k<\ell.
$$
Let 
$$
r = \frac{k}{\gcd(k,\ell)}, \qquad t = \frac{\ell}{\gcd(k,\ell)}. 
$$ 
Clearly, 
$$
1\le r < t, \qquad \gcd(r,t)=1. 
$$
From~\eqref{eq:mult kel}, we have 
\begin{equation}
\label{eq:mult rs}
\(f^{(m)}(\alpha)\)^r = \eta\(f^{(n)}(\alpha)\)^t
\end{equation} 
for some $\eta\in \U\cap \K$. 

Since $\gcd(r,t)=1$, we can now find some integers $a$ and $b$ with 
$ar + bt =1$. Hence, using~\eqref{eq:mult rs} we get 
\begin{equation}
\label{eq:s pow}
f^{(m)}(\alpha) = \eta^a \(\(f^{(n)}(\alpha)\)^a \(f^{(m)}(\alpha)\)^b\)^t.
\end{equation}

Now, given a polynomial $g\in \K[X]$ we can write it as 
$$
g(X) = \frac{1}{D} G(X)
$$
with $G\in \ZK[X]$ and a positive integer $D$, both are uniquely defined by the
minimality condition on $D$. 
Then we use  $\cS_g$ to denote the subset of places of $\K$, which consists of all the infinite places and all the finite places corresponding to  the prime ideal divisors of  $D$ and of the leading coefficient of $G$.

One can easily verify that 
\begin{equation}
\label{eq:Sm S}
\cS_{f^{(m)}} \subseteq \cS_{f}, \qquad m =1,2, \ldots.
\end{equation}

Furthermore, for a prime ideal $\fp$ of $\ZK$ we use $v_\fp(\vartheta)$ 
to denote the (additive) valuation of $\vartheta \in \K$ at the place corresponding to $\fp$.  
We denote by $\ZKSf$  the set of $\cS_f$-integers in $\K$, that is, 
the set of $\vartheta \in \K$ with $v_\fp(\vartheta)\ge 0$ for any $\fp \not \in \cS_f$
(alternatively,  $v_\fp(\vartheta)< 0$ implies $\fp \in \cS_f$). 
Note that $f(\alpha)$ is in $\Z_{\K,\cS_f}$ for $\alpha \in \Z_{\K,\cS_f}$.

Now, we first assume that $m = n+1$. 
From~\eqref{eq:mult rs}, we have 
\begin{equation}
\label{eq:mult rsn}
\(f\(f^{(m-1)}(\alpha)\)\)^r = \eta\(f^{(m-1)}(\alpha)\)^t. 
\end{equation} 
Clearly, the polynomials $f(X)$ and $X$ are multiplicatively independent modulo constants. 
Applying Lemma~\ref{lem:exponent} to~\eqref{eq:mult rsn} and noticing that there are only finitely many roots of unity in $\K$, 
 we obtain that the exponents $r$ and $t$ are upper bounded in terms of $f$ and $\K$ only. 
 So,  there are only finitely many possible values of $f^{(m-1)}(\alpha)$ satisfying~\eqref{eq:mult rsn}. 
 Hence, we have 
 \begin{equation}
\label{eq:m=n+1}
\house{\strut f^{(m-1)}(\alpha)} \ll_{f,\K} 1.  
\end{equation}

We now assume that $m = n+2$. 
From~\eqref{eq:mult rs}, we have 
\begin{equation}
\label{eq:mult rsn2}
\(f^{(2)}\(f^{(m-2)}(\alpha)\)\)^r = \eta\(f^{(m-2)}(\alpha)\)^t. 
\end{equation} 
Since $f$ only has simple roots and its degree $d \ge 2$, the polynomials $f^{(2)}(X)$ and $X$ are multiplicatively independent modulo constants. 
As the above, using Lemma~\ref{lem:exponent}, we deduce that there are only finitely many possible values of $f^{(m-2)}(\alpha)$ satisfying~\eqref{eq:mult rsn2}. 
 Hence, we have 
$$
\house{\strut f^{(m-2)}(\alpha)} \ll_{f,\K} 1,
$$
which ensures that~\eqref{eq:m=n+1} holds in this case too.

In the following, we assume that $m > n+2$. 

We distinguish between two cases: one when~\eqref{eq:s pow} holds for 
$\alpha \in \ZKSf$ and one for $\alpha \not \in \ZKSf$.

\subsubsection*{Case~I: $\alpha \in  \ZKSf$} 
In this case, 
since $f^{(m)}(\alpha) \in  \ZKSf$ and $\eta \in \U\cap \ZK$, by~\eqref{eq:s pow} we also have 
$$
\(f^{(n)}(\alpha)\)^a \(f^{(m)}(\alpha)\)^b \in  \ZKSf 
$$
(because it is in $\K$ and its $t$-th power is in $\ZKSf$). 

Write 
$$
f^{(m)}(\alpha)  = f\(f^{(m-1)}(\alpha) \),
$$
or as
$$
f^{(m)}(\alpha)  = f^{(2)}\(f^{(m-2)}(\alpha) \)
$$
if $(d,t) = (2, 2)$. 
By a result of B{\'e}rczes,   Evertse and Gy{\" o}ry~\cite[Theorem~2.3]{BEG}, we obtain that the exponent $t \ge 2$ in~\eqref{eq:s pow} is upper bounded in terms of $f$ and $\K$. 
Applying now~\cite[Theorems~2.1 and~2.2]{BEG} to~\eqref{eq:s pow}, we conclude that 
$$
\h\(\strut f^{(m-1)}(\alpha)\)\ll_{f,\K} 1.  
$$
Since $f^{(m-1)}(\alpha)\in\K$, by the Northcott theorem we conclude that there are only finitely many possible values of $f^{(m-1)}(\alpha)$, and thus the inequality~\eqref{eq:m=n+1} holds again.

\subsubsection*{Case~II: $\alpha \not \in  \ZKSf$} 
We choose a prime ideal $\fp$ of $\Z_\K$ with 
$$
\fp\not \in \cS_f \mand v_\fp(\alpha) < 0.
$$
Then, using~\eqref{eq:Sm S}, we see that $\fp\not \in \cS_{f^{(m)}}$ and thus
we easily derive
\begin{equation}
\label{eq:ord fm}
 v_\fp\(f^{(m)}(\alpha)\) = d^m  v_\fp(\alpha).
\end{equation}
Indeed, let 
$$
f^{(m)}(X) = \frac{1}{E} \sum_{j=0}^{d^m}A_j X^j
$$
where $E, A_{d^m}, \ldots, A_0 \in \ZK$ with $ v_\fp\( A_{d^m}E\) = 0$.
Then 
\begin{equation}
\label{eq:ord a fm}
\begin{split}
 v_\fp\(\alpha^{-d^m}f^{(m)}(\alpha)\) &= 
 v_\fp\(E^{-1}\(A_{d^m} + \sum_{j=0}^{d^m-1}A_j \alpha^{j - d^m}\)\)\\
& =  v_\fp\(A_{d^m}  + \sum_{j=0}^{d^m-1}A_j \alpha^{j - d^m}\) = 0,
\end{split}
\end{equation}
since $v_\fp\(A_{d^m}\) = 0$, while
\begin{align*}
v_\fp\( \sum_{j=0}^{d^m-1}A_j \alpha^{j - d^m}\) & = v_\fp\(\alpha^{-1}\)
+ v_\fp\( \sum_{j=0}^{d^m-1}A_j \alpha^{j - d^m+1}\) \\
&\ge v_\fp\(\alpha^{-1}\)
+ \min_{0 \le j \le d^m-1}v_\fp\(A_j \alpha^{j - d^m+1}\) \\
& \ge v_\fp\(\alpha^{-1}\)> 0.
\end{align*}
Clearly~\eqref{eq:ord a fm} implies~\eqref{eq:ord fm}. 
 
Hence, recalling~\eqref{eq:mult rs}, we conclude that 
$$
 r d^m  v_\fp(\alpha) =  t d^n  v_\fp(\alpha).
$$
Since $1 \le r < t$, $\gcd(r,t)=1$ and $v_\fp(\alpha)\ne 0$, we conclude that 
$r=1$ and $t = d^{m-n}$, and moreover 
\begin{equation}
\label{eq:mult r1}
f^{(m)}(\alpha) = \eta\(f^{(n)}(\alpha)\)^{d^{m-n}}. 
\end{equation} 
We enlarge $\K$ to a field $\L$ such that it contains all $d^3$-th roots of roots of unity in $\K$. Noticing $m>n+2$,  we obtain from~\eqref{eq:mult r1}
$$
f^{(m)}(\alpha) =w^{d^3}
$$
for some $w\in\L$.
Since $f$ is of degree $d \ge 2$ and has only simple roots,  the curve $f(X)=Y^{d^3}$ is 
smooth of genus greater than $1$, see~\cite[Theorem A.4.2.6]{HinSil}.   By the celebrated result of Faltings~\cite{Fa83},
 we know that the curve $f(X)=Y^{d^3}$ has only finitely many $\L$-rational points, and in particular this implies that we have the inequality~\eqref{eq:m=n+1} again. 

Therefore, we can always choose the constant $N$ above to be sufficiently large (depending on $f,\K$) such 
that~\eqref{eq:m=n+1} can be written as
$$
\house{\strut f^{(m-1)}(\alpha)} < N.
$$
This together with Lemma~\ref{lem:itergrowth Arch} (ii) implies that $\house{\strut \alpha} \le N$, 
which however contradicts  the assumption $\house{\strut \alpha} > N$.
Therefore, we see that there is no  $\alpha \in \K \setminus \cT(N,\K)$ which satisfies~\eqref{eq:mult kel}.
This   completes the proof. 
\end{proof}

\begin{remark}
{\rm 
We note that the assumption that $f^{(2)}$ has no multiple roots when $d=2$ in Theorem~\ref{thm: 2multdep-arb} is equivalent to imposing that the critical point of $f$ is not a root of $f^{(2)}$.
}
\end{remark}

We remark that since Faltings' theorem is ineffective, the result in Theorem~\ref{thm: 2multdep-arb} is also ineffective. 

We note that Theorem~\ref{thm: 2multdep-arb} also implies that if we fix a non-preperi\-o\-dic point $\alpha\in\K$, then there are only finitely many $n\ge 1$ such that $\alpha$ and $f^{(n)}(\alpha)$ are multiplicatively dependent. However, such a conclusion can be easily obtained in much greater generality as in the next result.

We also note that if $\varphi\in\K[X]$ is of degree at least $2$ and not special, then by~\cite[Theorem~2]{DZ} there are only finitely many preperiodic points $\alpha\in\Kc$ of $\varphi$. Thus, we look at multiplicative relations in the orbits of non-preperiodic points.

\begin{theorem}
\label{thm: 2multdep-init}
Let $\varphi \in \K(X)$ be of degree $d \ge 2$ and not of the form $ \beta X^{\pm d}$ with $\beta \in \K^*$
and let $\alpha\in\ov\Q\setminus\U$ be non-preperiodic  for $\varphi$. 
Then, there are only finitely many positive integers $n$ such that $\alpha$ and $\varphi^{(n)}(\alpha)$ are multiplicatively dependent.
\end{theorem}

\begin{proof} 
First, we can extend the field $\K$ to include also the element $\alpha$, and thus also all elements $\varphi^{(n)}(\alpha)$, $n\ge 1$.

Assume that there is a multiplicative relation, that is, there exists integers $k_1,k_2$, not both zero, such that 
\begin{equation}
\label{eq:alpha n}
\alpha^{k_1}\varphi^{(n)}(\alpha)^{k_2}=1
\end{equation}
for some $n\ge 1$. Since $\alpha$ is not a root of unity, we have $k_2 \ne 0$ in~\eqref{eq:alpha n}.  

If we denote by $\Gamma$ the group generated by $\alpha$ in $\K^*$, 
the equation~\eqref{eq:alpha n} is equivalent to $\varphi^{(n)}(\alpha)\in\ov\Gamma$, where $\ov\Gamma$ is the so-called division group of $\Gamma$ defined by 
$$
\ov\Gamma = \{a \in \ov\Q^*: \, \textrm{$a^m \in \Gamma$ for some positive integer $m$} \}. 
$$
We now apply a version~\cite[Corollary~2.3]{OstShp} (which is based on~\cite{KLSTYZ}) 
to conclude that there are finitely many values $\varphi^{(n)}(\alpha)$ satisfying~\eqref{eq:alpha n} when $n$ varies. 
Note that although~\cite[Corollary~2.3]{OstShp}  is stated only for polynomials,  the same argument works for rational functions satisfying our condition as well, and also under a more relaxed condition. 
Indeed, let $S$ be a finite set of places of $\K$, including the Archimedean ones, such that $|\alpha|_v=1$ for any  $v\not\in S$. 
So, $|\beta|_v=1$ for any $\beta \in\Gamma$ and $v\not\in S$. Now, for any $\gamma \in\ov\Gamma$, that is $\gamma^m\in\Gamma$ for some positive integer $m$, 
we also  have $|\gamma|_v=1$ for any $v\not\in S$. Thus, $\varphi^{(n)}(\alpha)\in R_S^*$, 
where $R_S^*$ is the ring of $S$-units in $\K$, and now the finiteness of the corresponding set of values $\varphi^{(n)}(\alpha)$ follows 
from~\cite[Proposition~1.6,~(a)]{KLSTYZ}. 

Since $\alpha$ is non-preperiodic, the points in the orbit $\Orb_\varphi(\alpha)$ are all distinct. 
This implies that there are finitely many $n\ge 1$ satisfying~\eqref{eq:alpha n}. So, we complete the proof. 
\end{proof}

\section*{Acknowledgement}
The authors 
 are grateful to David Masser  for valuable discussions around Lemma~\ref{lem:mult-Kab} and 
 to Gabriel Dill for pointing out some errors in an early version of the paper. 
They also would like to thank the referee for a careful reading and valuable comments. 

The first, the third and the fourth authors are also grateful to the Fields Institute for the hospitality and generous 
support during the {\it Thematic Program on Unlikely Intersections, Heights, and Efficient Congruencing\/}, where some parts of this paper were developed.

During the preparation of this work, A.~O. was supported by the UNSW FRG Grant PS43704 and 
by the ARC Grant~DP180100201, M.~S. was supported by the Macquarie University Research Fellowship,  
and I.~S.   was  supported   by the ARC Grants~DP170100786 and DP180100201.

\end{document}